\documentclass[12pt,draft]{amsart}
\usepackage{amscd,amssymb}
\usepackage{typearea}

\newcounter{thmcounter}

\newtheorem{lemma}{Lemma}

\newtheorem*{corollary*}{Corollary}

\newtheorem{theorem}[thmcounter]{Theorem}
\newtheorem*{theorem*}{Theorem}

\newcommand{\IP}{{\bf P}}
\newcommand{\IG}{{\bf G}}
\newcommand{\IC}{{\bf C}}
\newcommand{\IH}{{\bf H}}
\newcommand{\IR}{{\bf R}}
\newcommand{\IQbar}{\overline{\bf Q}}
\newcommand{\IZ}{{\bf Z}}

\newcommand{\IQ}{\mathbf{Q}}

\newcommand{\ssm}{\smallsetminus}

\newcommand{\re}[1]{\rm Re({#1})}
\newcommand{\height}[1]{h({#1})}

\newcommand{\cO}{{\mathcal O}}
\newcommand{\cF}{{\mathcal F}}

\begin{document}

\title{Singular Moduli that are Algebraic Units}
\author[P. Habegger]{P. Habegger}
\address{{Fachbereich Mathematik,} 
{Technische Universit\"at Darmstadt,}
 Schlossgartenstra\ss e 7,
{64289 Darmstadt, Germany}}
\email{habegger@mathematik.tu-darmstadt.de}


\begin{abstract}
  We prove that  only finitely many $j$-invariants of elliptic
  curves with complex multiplication are
  algebraic units. A rephrased and generalized version of this result
resembles 
 Siegel's Theorem on integral points of algebraic curves.
\end{abstract}
\maketitle
\section{Introduction}

A singular modulus  is 
 the $j$-invariant of an elliptic curve with
complex multiplication; we treat them as complex numbers in this note.  
They are precisely the values 
of  Klein's modular function
$j:\IH\rightarrow \IC$  at imaginary quadratic arguments; 
here $\IH$ denotes the upper half-plane in $\IC$. 
For example, $j(\sqrt{-1}) = 1728$.
Singular moduli are algebraic integers
and their entirety
is stable under ring automorphisms of $\IC$. 
We refer to Lang's book \cite{Lang:elliptic} for such classical facts. 

At the AIM workshop on unlikely intersections in algebraic groups and
Shimura varieties in Pisa, 2011 David Masser, motivated by \cite{BiMaZa13}, asked if there are only
finitely many singular moduli that are algebraic units. Here we
provide a positive answer to this question. 

\begin{theorem}
\label{thm:junit}
At most finitely many singular moduli are algebraic units. 
\end{theorem}

Our theorem relies on several tools: Liouville's inequality from diophantine
approximation,  Duke's Equidistribution Theorem
\cite{Duke:hyperbolic}, its generalization due to Clozel-Ullmo
\cite{CU:Equidistribution},
and Colmez's lower bound for the Faltings height of an elliptic curve
with complex multiplication \cite{Colmez} supplemented by work of 
Nakkajima-Taguchi \cite{NT}.

A numerical computation involving {\tt sage} reveals  that no singular modulus of degree at
most $100$ over the rationals is an algebraic unit. There may be no
such units at all. 
Currently, there is no way to be sure  as
 Duke's Theorem is not known to be effective.

Below, we formulate and prove a general finiteness theorem reminiscent to
Siegel's Theorem on integral points on curves. 
We will see in particular  that there are only finitely singular moduli $j$
such that $j+1$ is a unit.  Now there are examples, as $j((\sqrt{-3}+1)/2)=0$ is a singular modulus.

Suppose that $X$ is a geometrically irreducible, smooth, projective curve defined
over a number field $F$. 
We write $F[X\ssm C]$ for the rational functions on $X$ that
are regular outside of a finite subset $C$ of $X(F)$. 
Let $\cO_F$ be the ring of algebraic integers of $F$.
A subset $M\subset X(F)\ssm C$ is called
quasi-integral with respect to $C$ if for any $f\in  F[X\ssm C]$ there
exists $\lambda\in  F\ssm\{0\}$ such that $\lambda f(M) \subset \cO_F$.
By clearing denominators one sees  that  quasi-integral
sets remain so after adding finitely many $F$-rational points. 
Siegel's Theorem, cf. Chapter 7 \cite{Serre:MordellWeil},  states that a quasi-integral sets is finite if
$C\not=\emptyset$ and the genus of $X$ is positive, or if $\# C\ge 3$.

Our extension of Theorem \ref{thm:junit} will deal with the question
of finiteness for  quasi-integral
sets of special points on modular curves. Special points generalize
singular moduli, we provide a definition below.  Only finitely many singular moduli
are rational over a fixed number field. Thus we adapt the notion
of quasi-integrality in the following way.  Let $\overline F$ be an algebraic
closure of $F$ and $\cO_{\overline F}$ the ring of algebraic integers in
$\overline F$.
We again work with a finite set $C\subset X(\overline F)$. 
A subset $M\subset X(\overline F)\ssm C$ is called
quasi-algebraic-integral with respect to $C$ if for all $f\in
\overline F[X\ssm C]$ there
is $\lambda\in \overline F\ssm\{0\}$ such that $\lambda f(M) \subset \cO_{\overline{F}}$.

Let us recall some classical facts about modular curves. 
 Let $\Gamma$ be a subgroup of ${\rm SL}_2(\IZ)$ that
contains the kernel of the reduction homomorphism ${\rm SL}_2(\IZ)\rightarrow {\rm
  SL}_2(\IZ/N\IZ)$ for an $N\ge 1$. These subgroups are
called congruence subgroups of ${\rm SL}_2(\IZ)$. They act on  $\IH$,
as does any subgroup of ${\rm SL}_2(\IR)$, by fractional linear transformations. 
The quotient $\IH/\Gamma$ can be equipped
with the structure of an algebraic curve $Y_\Gamma$ defined over a
number field $F$. This algebraic curve has  a natural compactification
$X_\Gamma$, which is a geometrically irreducible, projective, smooth curve over
$F$.
The points of $X_\Gamma \ssm Y_\Gamma$ are called the cusps of
$Y_\Gamma$. 
We remark that $Y(1) = Y_{{\rm SL}_2(\IZ)}$ is the affine
line, that the compactification is $\IP^1$, and that there is a single
cusp $\infty$.  The natural map $\phi:Y_\Gamma \rightarrow Y(1)$
is algebraic. A point of $Y_\Gamma(\overline F)$ is called special if
it maps to a singular modulus under $\phi$.

\begin{theorem}
\label{thm:modularsiegel}
  Let $\Gamma\subset {\rm SL}_2(\IZ)$ be a congruence subgroup and $F\subset\IC$
  a number field over which $Y_\Gamma$ is defined. Let $C\subset
  X_\Gamma(\overline F)$ be a finite set containing a point that is not a cusp of
  $Y_\Gamma$. 
Any set of special points in $Y_\Gamma(\overline F)$ that is
quasi-algebraic-integral with respect to $C$ is finite. 
\end{theorem}

We require $C$ to contain a non-cusp for good reason. Indeed, as singular moduli
are algebraic integers, their totality is a
quasi-algebraic-integral subset of $Y(1)(\IQbar)$ with respect to 
$C = \{\infty\}$.
We recover Theorem \ref{thm:junit} from Theorem \ref{thm:modularsiegel}
on taking $\Gamma = {\rm SL}_2(\IZ)$ and   $C = \{0,\infty\}$.

The proof of Theorem \ref{thm:modularsiegel} relies on the same basic
strategy  as Theorem \ref{thm:junit}. However, instead of the
Liouville inequality we
require David and Hirata-Kohno's  sharp lower bound for linear forms
in elliptic logarithms \cite{DHK}.  
Earlier, Masser and others obtained lower bounds in this setting
after A. Baker's initial work 
on  linear forms in classical logarithms. 

Our theorems  are reminiscent to M. Baker, Ih, and Rumely's
result \cite{BaIhRu} on  roots of unity that are $S$-integral relative
to a divisor of $\IG_m$. Indeed, both finiteness results are based on
an equidistribution statement. 
However, the Weil height of a root of unity is
zero, whereas the  height of a singular modulus can be arbitrarily
large. 
Indeed, the quality of Colmez's  growth estimate for the Faltings height
 plays a crucial role in our argument. 
Moreover, finiteness need not hold in the multiplicative setting 
 if the support of the divisor consists of roots of unity. 
This is in contrast to Theorem \ref{thm:junit} where the support of
the corresponding divisor is the singular modulus $0$. 
Finally, our work considers only  the case where $S$ consists only
of the Archimedean places whereas M. Baker, Ih, and Rumely also allow finite
places.

The author would like to thank the organizers of the AIM workshop in
Pisa, 2011 for providing a stimulation environment. He also thanks Su-ion Ih for helpful
remarks concerning his paper with M. Baker and Rumely.  

\section{Unitary Singular Moduli} 
\label{sec:usm}

In this section $c_1,c_2,\ldots$ denote positive and absolute
constants.

Let $K$ be a number field. A finite place $\nu$ of $K$ is a non-Archimedean
absolute value that restricts to the $p$-adic absolute value on $\IQ$
for some prime $p$. 
With this normalization we have $|p|_\nu = 1/p$. 
The completion of $K$ with respect to $\nu$ is
a field extension of degree $d_\nu$ of the completion of $\IQ$ with
 respect to the  $p$-adic absolute value. 
Let $J$ be an algebraic number in a number field $K$.
The absolute logarithmic Weil height of $J$, or just height for short, is 
\begin{equation*}
  \height{J} = \frac{1}{[K:\IQ]}\left(\sum_{\sigma}
\log\max\{1,|\sigma(J)|\} + \sum_{\nu} d_\nu \log\max\{1,|J|_\nu\}\right)
\end{equation*}
where  $\sigma$ runs over all field embeddings
$\sigma:K\rightarrow\IC$
and $\nu$ runs over all finite places of $K$.
It is well-known that $\height{J}$ does not change when replacing $K$
by another number field containing $J$. For this and other facts
on heights we refer to Sections 1.5 and 1.6 of Bombieri and Gubler's book 
\cite{BG}. 

We state a height lower bound for singular moduli that follows easily from results of   Colmez and
Nakkajima-Taguchi. 

\begin{lemma}
\label{lem:colmez}
Let  $J$ 
be a singular modulus attached to an elliptic curve 
whose endomorphism ring is an order with discriminant
$\Delta<0$. 
Then 
  \begin{equation}
\label{eq:colmez}
    \height{J} \ge c_2 \log|\Delta| - c_3.
  \end{equation}
\end{lemma}
\begin{proof}
We  write $\Delta = \Delta_0 f^2$ where $\Delta_0<0$ is a
fundamental discriminant and $f$ is the conductor of the endomorphism
ring of $E$, an elliptic curve attached to $j$.
  Colmez \cite{Colmez} proved (\ref{eq:colmez})
with $\height{J}$ replaced by the stable Faltings height of $E$ when $\Delta$
is  a fundamental discriminant, i.e. if $f=1$. 
For $f>1$ Nakkajima and Taguchi \cite{NT}  found 
that one must add
\begin{equation*}
\frac 12    \log f - \frac 12 \sum_{p|f} e_f(p)\log p
\end{equation*}
to the Faltings height;
here the sum runs over prime divisors $p$ of $f$ and 
\begin{equation*}
  e_f(p) = \frac{1-\chi(p)}{p-\chi(p)}\frac{1-p^{-n}}{1-p^{-1}}
\end{equation*}
if $p^n\mid f$ but $p^{n+1}\nmid f$ and $\chi(p)$ is Kronecker's
symbol
$(\frac{\Delta_0}{p})$. Now $\sum_{p|f}e_f(p)\log p \le c_1
\log\log\max\{3,f\}$
by the arguments in the proof of Lemma 4.2 \cite{habegger:wbhc}.
Therefore, the Faltings height of $E$ is bounded from below
logarithmically in terms of $|\Delta_0 f^2|=|\Delta|$. 

Silverman's  Proposition 2.1 \cite{Silverman:HEC} allows us to 
replace the Faltings height by $\height{J}$
at the cost of adjusting the constants. 
\end{proof}

Our strategy to prove Theorem \ref{thm:junit} is as follows.
Let $J$ and $\Delta$ be as in Lemma \ref{lem:colmez}.
Assume in addition that $J$ is an algebraic
 unit.  We will find an upper bound for $\height{J}$ that contradicts
 the previous lemma for sufficiently large $|\Delta|$.
This will leave us with only  finitely many $\Delta$ and hence
 finitely many  $J$, as we will see.

The norm of $J$ is $\pm 1$ and the finite places do not contribute to
the height of the algebraic integer $J$.
Thus we can rewrite
\begin{equation}
\label{eq:rewrittenheight}
  \height{J} =
 \frac{1}{D} \sum_{|\sigma(J)|>1}
  \log|\sigma(J)|
=
-\frac{1}{D} \sum_{|\sigma(J)|<1}
  \log|\sigma(J)|
\end{equation}
where $D=[\IQ(J):\IQ]$ and where the sums run over  field embeddings
$\sigma:\IQ(J)\rightarrow\IC$.

For each $\sigma$ we have $\sigma(J)=j(\tau_\sigma)$ for some
$\tau_\sigma$  in the  classical fundamental domain 
\begin{equation*}
  \cF = \left\{\tau\in \IH;\,\, \re{\tau}\in (-1/2, 1/2],\, |\tau|\ge 1\text{
    and }\re{\tau}\ge 0 \text{ if }|\tau|=1\right\}
\end{equation*}
 of the action of ${\rm SL}_2(\IZ)$ on $\IH$. 

To bound the right-hand side of (\ref{eq:rewrittenheight}) from above we must control those
conjugates $\sigma(J)$ that are small in modulus. 
Let  $\epsilon \in (0,1]$ be a parameter that is to be determined;
the $c_i$ will not depend on $\epsilon$.  We define 
\begin{equation*}
  \Sigma_\epsilon = \{\tau\in \cF;\,\, |j(\tau)|<\epsilon \}.
\end{equation*}
The field embeddings that contribute most to the height of $J$ are in 
\begin{equation*}
  \Gamma_\epsilon = \{\sigma:\IQ(J)\rightarrow \IC;\,\, \tau_\sigma \in
  \Sigma_\epsilon \}.
\end{equation*}
We  estimate  their  number 
using equidistribution in the next lemma. 

\begin{lemma}
\label{lem:equi}
  We have $\#\Gamma_\epsilon  \le c_6 \epsilon^{2/3} D$
if $D$ is sufficiently large with respect to $\epsilon$. 
\end{lemma}
\begin{proof}
Let $\mu$ denote the hyperbolic measure on $\cF$ with total mass $1$, i.e.
\begin{equation}
\label{eq:measurebound}
  \mu(\Sigma) = \frac{3}{\pi}\int_{x+yi\in \Sigma} \frac{dxdy}{y^2}
\end{equation}
for a measurable subset $\Sigma \subset \cF$.
Duke \cite{Duke:hyperbolic} 
proved that the $\tau_\sigma$ are equidistributed with respect to
$\mu$ as $\Delta\rightarrow -\infty$ runs over fundamental discriminants. 
For general discriminants equidistribution follows from a 
result  of
Clozel and Ullmo \cite{CU:Equidistribution}.
So $|\#\Gamma_\epsilon/D - \mu(\Sigma_\epsilon)|
\rightarrow 0$ as $\Delta\rightarrow -\infty$. 
To prove the lemma we will bound $\mu(\Sigma_\epsilon)$ in terms of
$\epsilon$.

Let $\zeta$ be the unique root of unity in $\IH$ of order $6$.
By
Theorem 2, Chapter 3 \cite{Lang:elliptic}  Klein's modular function has a triple zero at
$\zeta$ and at $\zeta^2$ and does not vanish anywhere  else on
$\overline\cF$,
the closure of $\cF$ in $\IH$.
So $\tau\mapsto
j(\tau)(\tau-\zeta)^{-3}(\tau-\zeta^2)^{-3}$
does not vanish on $\overline\cF$. 
Using the $q$-expansion
\begin{equation*}
  j(\tau) = \frac 1q + 744 + 196884q + \cdots \quad\text{with}\quad
q=e^{2\pi \sqrt{-1} \tau}
\end{equation*}
we see that
 $|j(\tau)|$ grows exponentially in $|\tau|$  if
$|\tau|\rightarrow\infty$ in $\overline{\cF}$. 
So
\begin{equation}
\label{eq:jzlb}
 |j(\tau)| \ge c_4 |\tau-\zeta|^3 |\tau-\zeta^2|^3 \ge \frac{c_4}{8} 
\min \{|\tau-\zeta|,|\tau-\zeta^2|\}^3
\quad\text{for all}\quad \tau\in\overline{\cF}
\end{equation}
where  $\max\{|\tau-\zeta|,|\tau-\zeta^2|\}\ge |\zeta-\zeta^2|/2=1/2$
was used in the second inequality.
Because the imaginary part of an element in $\cF$ 
is at least $\sqrt{3}/2$
we can use (\ref{eq:measurebound}) to estimate
  $\mu(\Sigma_\epsilon) \le c_5 \epsilon^{2/3}$.
\end{proof}

Using this lemma with (\ref{eq:rewrittenheight})
we can bound the height of $J$ from above as
  \begin{alignat}1
\nonumber
    \height{J} &= -\frac 1D\left( \sum_{\substack{ |\sigma(J)|<\epsilon }}      
  \log |\sigma(J)|
+\sum_{ \epsilon \le |\sigma(J)|<1 }
\log|\sigma(J)| \right)\\
\label{eq:heightjup1}
&\le c_6  \epsilon^{2/3} \max_{|\sigma(J)|<\epsilon}  \log(|\sigma(J)|^{-1})
+ |\log\epsilon|.
  \end{alignat}

Soon we will use Liouville's inequality from diophantine approximation to bound 
$|j(\tau_\sigma)|$ from below if $\sigma \in \Gamma_\epsilon$.
To do this we first require a bound for  the height of $\tau_\sigma$.

\begin{lemma}
\label{lem:htaubound}
  Each $\tau_\sigma$ is imaginary quadratic and
$\height{\tau_\sigma}\le  \log\sqrt{|\Delta|}$. 
\end{lemma}
\begin{proof}
We abbreviate $\tau=\tau_\sigma$
and decompose 
$\Delta=\Delta_0 f^2$ as in the proof of Lemma \ref{lem:colmez}. 
The endomorphism ring mentioned in the said lemma  can be identified with $\IZ+\omega f\IZ\subset \IC$ where
$\omega = (\sqrt{\Delta_0}+\Delta_0)/2$. 
This ring acts on the  lattice $\IZ+\tau\IZ$. So 
there exist $a,b,c,d\in\IZ$ with $\omega f = a+b\tau$,
$\omega f \tau = c+d\tau$ and $b\not=0$. We substitute the first
equality into  the second one and obtain
\begin{equation}
\label{eq:tauequation}
  b\tau^2 + (a-d)\tau - c=0. 
\end{equation}
Of course,  $\tau$ is imaginary quadratic. 
We observe that $\omega f$ is a root of $T^2 -(a+d)T + ad-bc$. The 
 discriminant of this quadratic polynomial is $(a+d)^2-4(ad-bc) =
 (\omega-\overline\omega)^2f^2
= \Delta $.  Hence
$\tau = ( -(a-d)\pm \sqrt{\Delta})/2b$
and therefore $|\tau|^2  =((a-d)^2 +|\Delta|)/(2b)^2$.

As $\tau$ lies in $\cF$ we have $|\tau|\ge 1$ and
$|\re{\tau}|\le 1/2$. The second inequality implies $|a-d|\le
|b|$ and hence $|\tau|^2\le (b^2 + |\Delta|)/(2b)^2$.
By Proposition 1.6.6  \cite{BG} the value $2\height{\tau}$ is at most  the logarithmic Mahler measure
of $bT^2+(a-d)T-c$. So $2\height{\tau}
\le\log(|b||\tau|^2) \le \log(|b|/4 + |\Delta|/(4|b|))$. 
 The
imaginary part of $\tau$ is at least $\sqrt{3}/2$ and so
$|b|\le \sqrt{|\Delta|/3}$.
As $x\mapsto x + |\Delta|/x$ is decreasing on
$[1,\sqrt{|\Delta|}]$ we conclude
 $2\height{\tau}\le \log((1+|\Delta|)/4) \le \log|\Delta|$.
\end{proof}

Now we use Liouville's inequality to bound the conjugates of $J$ away
from zero. 

\begin{lemma}
We have $\log |\sigma(J)| \ge -c_8\log |\Delta|$
for any $\sigma:\IQ(J)\rightarrow\IC$. 
\end{lemma}
\begin{proof}
We retain the notation of the proof of Lemma \ref{lem:equi} and assume $|\tau_\sigma-\zeta| \le |\tau_\sigma-\zeta^2|$;
the reverse case is similar. 
According to (\ref{eq:jzlb}) we have
\begin{equation}
\label{eq:jlowerbound}
|\sigma(J)| = |j(\tau_\sigma)|
\ge c_7 |\tau_\sigma-\zeta|^3.
\end{equation}
 We also remark
$\tau_\sigma\not=\zeta$ since   $\sigma(J)\not=0=j(\zeta)$.
 Liouville's inequality, Theorem 1.5.21 \cite{BG}, tells us
\begin{equation*}
  -\log |\tau_\sigma-\zeta| \le
  [\IQ(\tau_\sigma,\zeta):\IQ](\height{\tau_\sigma}+\height{\zeta}+\log
  2).
\end{equation*}
But $\tau_\sigma$ and $\zeta$ are imaginary quadratic, so
$[\IQ(\tau_\sigma,\zeta):\IQ]\le 4$. Moreover,  $\height{\zeta}=0$ as
$\zeta$ is  a root of unity. The bound for  $\height{\tau_\sigma}$
from
 Lemma \ref{lem:htaubound} yields
\begin{equation*}
-\log |\tau_\sigma-\zeta| \le 4\log(2\sqrt{|\Delta|}).
\end{equation*}
The lemma now follows from  $|\Delta|\ge 3$ and
(\ref{eq:jlowerbound}).
\end{proof}

\begin{proof}[Proof of Theorem \ref{thm:junit}]
We will see soon how to fix $\epsilon$ in terms of the $c_i$. 
By a classical result of Heilbronn and Hecke there are only finitely
many singular moduli whose degree over $\IQ$ are bounded by a
prescribed constant. 
So there is no loss of generality if we assume that $D$ is large enough as in  Lemma \ref{lem:equi}

We use the previous lemma to bound the first term 
in (\ref{eq:heightjup1}) from above. Thus
\begin{equation*}
  \height{J} \le c_6c_8\epsilon^{2/3} \log |\Delta| +
  |\log\epsilon|.
\end{equation*}
We fix $\epsilon$ to
  satisfy $c_6c_8\epsilon^{2/3} < c_2/2$
where $c_2$ comes from the height lower bound in Lemma \ref{lem:colmez}.
With this choice we conclude that $|\Delta|$ is bounded from above by an absolute
constant. 
By Lemma \ref{lem:htaubound} and Northcott's Theorem there are only finitely many
possible $\tau_\sigma$ and thus only finitely many possible $J$.
\end{proof}

\section{Proof of Theorem \ref{thm:modularsiegel}}

We begin by stating a special case of David and Hirata-Kohno's deep lower bound
for linear forms in $n$ elliptic logarithms if 
 $n=2$ and when the 
elliptic logarithms are periods.

Let $E$ be an elliptic curve defined over
a number field in $\IC$. 
We fix a Weierstrass equation for $E$ with coefficients in the said
number field
and 
 a Weierstrass-$\wp$ function that induces a uniformization
$\IC\rightarrow E(\IC)$. This is a group homomorphism whose kernel
$\omega_1\IZ+\omega_2\IZ$ is a discrete subgroup of $\IC$. 
We start numbering constants anew.

\begin{lemma}
\label{lem:DHK}
Let $d\ge 1$.
  There exists a constant $c_1>0$ depending on $E,d,$ the choice
  of Weierstrass equation, and the choice $\omega_{1,2}$ with the following property. 
Suppose $\alpha,\beta\in\IC$ are algebraic over $\IQ$ of degree at most $d$
and
$\max\{1,\height{\alpha},\height{\beta}\}\le \log B$ for some real
number $B>0$.
If $\alpha\omega_1+\beta\omega_2\not=0$, then 
\begin{equation}
\label{eq:DHKlb}
  \log |\alpha\omega_1+\beta\omega_2|\ge -c_1 \log B.
\end{equation}
\end{lemma}
\begin{proof}
  This follows from Theorem 1.6 \cite{DHK}. 
\end{proof}

In our application, $\log B$  from (\ref{eq:DHKlb}) will be
approximately
$\log|\Delta|$ and will  compete directly with
the logarithmic lower bound in Lemma \ref{lem:colmez}. It  is thus
 essential that David and Hirata-Kohno's inequality is
logarithmic in $B$. 
A worse dependency such as  $-c_1 (\log B)(\log\log B)$
would not suffice. 

We further distill this result into a formulation adapted to our application.

\begin{lemma}
\label{lem:dioapprox}
  Suppose $\eta\in\IH$ such that $j(\eta)$ is an algebraic number. There exists a
  constant $c_2>0$ which may depend on $\eta$  with the following property. If $\tau\in\IH$ is
  imaginary quadratic with $\max\{1,\height{\tau}\}\le \log B$ for
  some real number $B>0$ and
  if $\tau \not=\eta$, then
  \begin{equation*}
    \log |\tau-\eta |\ge -c_2 \log B. 
  \end{equation*}
\end{lemma}
\begin{proof}
The algebraic number  $j(\eta)$ is the $j$-invariant on an elliptic
curve as introduced before Lemma \ref{lem:DHK}. We may assume that the periods
$\omega_{1,2}$  satisfy $\eta=\omega_2/\omega_1$.
As $\tau\not=\eta$ the lemma above
 with $\alpha=\tau$ and $\beta=-1$ implies $\log |\tau\omega_1 -\omega_2  |\ge
-c_1\log B$. We subtract $\log|\omega_1|$ and obtain
$\log|\tau-\eta|\ge -c_1\log B - \log|\omega_1|$. 
This lemma follows with an appropriate $c_2$ as $B\ge 2$.
\end{proof}

Let us suppose that $\Gamma,F,$ and $C$ are as in  Theorem
\ref{thm:modularsiegel}. 
We recall that $\phi$
is the natural morphism $Y_\Gamma\rightarrow Y(1)$ and may regard it
as an element in the function field of $X$. 
We abbreviate $X = X_\Gamma$. In the following we enlargen $F$ to a
number field for which $X(F)$ contains $C$ and all poles of $\phi$. 

By   hypothesis  there is  $P_0\in C$ that is not a cusp of $Y_\Gamma$.
 We write $J_0\in F$ for the value of $\phi$ at
$P_0$. 

The Riemann-Roch Theorem provides a
non-constant, rational function $\psi\in F[X\ssm\{P_0\}]$
that vanishes at the
 poles of $\phi$. 
As $\psi$ is regular outside of $P_0$, it must have a pole at $P_0$.

The functions  $\phi$
and $\psi^{-1}$ are algebraically dependent, i.e.
 there is an irreducible polynomial $R\in F[U,V]$ with
$R(\phi,\psi^{-1})=0$. We observe $\deg_U R >0$. 

\begin{lemma}
\label{lem:metric}
There exists a constant $c_5\in (0,1]$ which depends only on $R$ with the
following property. 
Let $K\supset F$ be a number field and $|\cdot|$ an absolute
 value on $K$
that extends the Archimedean absolute  value on $\IQ$.
If $u\in K$ and $v\in K\ssm\{0\}$ with $R(u,v)=0$,
$    |v|  < c_5$, and $u\not=J_0$, then 
 $ \log | u - J_0|  <  (\log|v|)/(2\deg_U R)$.
\end{lemma}
\begin{proof}
In this proof $c_{3,4}>0$ depend only on $R$.
Let us write
$R = r_0 + (U-J_0)r_1+\cdots + (U-J_0)^e r_e$
where $e=\deg_U R$ with $r_i \in F[V]$ and $r_e\not=0$. 

We first claim that $r_e$ is constant. Indeed, otherwise it would
vanish at some $v$ which we may assume to be an element of $F$ after
possibly enlargening this number field. 
 As $\psi$ is non-constant, the irreducible polynomial $R$ is not divisible by a linear
 polynomial in $F[V]$.
So $e\ge 1$ and $r_i(v)\not=0$ for some $i$. Thus
 $X$ contains a point where $\psi^{-1}$ takes the value $v$ and
$\phi$ has a pole. This contradicts our choice of $\psi$. 

Without loss of generality we may assume $r_e=1$. Next we claim
$r_i(0)=0$ if $0\le i\le e-1$. If this were not the case, we could
find $J_0'\not=J_0$ with $R(J_0',0)=0$. This too is impossible by our choice of $\psi$.

Therefore, 
\begin{equation*}
  R = VQ + (U-J_0)^{e}
\end{equation*}
for some $Q\in F[U,V]$ with $\deg_U Q\le e-1$. 

Now let $u$ and $v$ be as in the hypothesis; we will see how to fix
$c_5\in (0,1]$
below. 
We have $ |u-J_0|^{e} = |v Q(u,v)|$ and
$|vQ(u,v)| \le  c_3  \max\{1,|u|\}^{e-1}$
as $|v|\le 1$.
If $|u|\ge\max\{1,2|J_0|\}$, then $|u-J_0| \ge |u|-|J_0|\ge
|u|/2$ and so 
$|u|^{e} \le 2^e c_3 |u|^{e-1}$. 
We find $|u|\le 2^e c_3$. In this case   $|u-J_0|^e =|vQ(u,v)|\le c_4
|v|$ for some $c_4\ge 1$. After adjusting $c_4$ the same bound holds
if $|u|< \max\{1,2|J_0|\}$. 
We set $c_5=c_4^{-2}$ and observe  
 $c_4|v|< |v|^{1/2}$ if $|v|<c_5$.
Thus $ |u-J_0|^e \le |v|^{1/2} <1$ and the lemma follows on
taking the logarithm.
\end{proof}

Let us now  prove Theorem \ref{thm:modularsiegel}. For this we  must verify
that a set $M\subset X(\overline F)$ of special points that is quasi-algebraic-integral
with respect to $C$ is finite.
By definition, $M$ cannot contain the pole of $\psi$ and without loss of
generality
we may assume that $M$ does not contain its zeros either.  
Finally, we may assume that $J_0\not\in \phi(M)$.
Say $\lambda\in F\ssm\{0\}$ with $\lambda
\psi(M)\subset\cO_{\overline{F}}$. 

We will use $c_6,c_7,\ldots$ to denote
positive constants that may depend on $\Gamma, F,C,\lambda,$ and $M$.

Suppose $P\in M$ and let $K\subset\overline F$ be a number field containing $F$ and the
values $\psi(P),\phi(P)$. 
After possibly shrinking $c_5$ we may assume $c_5<|\sigma(\lambda)|$
for all embeddings $\sigma:K\rightarrow\IC$. 
Then 
\begin{alignat*}1
  \height{\lambda \psi(P)}  &= 
\frac{1}{[K:\IQ]}\sum_{|\sigma(\lambda \psi(P))|> 1} \log |\sigma(\lambda
\psi(P))| \\
&\le \height{\lambda} + 
\frac{1}{[K:\IQ]}\left(\sum_{|\sigma(\lambda)|^{-1}< |\sigma(
  \psi(P))|\le c_5^{-1}} \log |\sigma(\psi(P))|
+
\sum_{|\sigma( \psi(P))|> c_5^{-1}} \log
|\sigma(\psi(P))|\right) \\
&\le  c_6 + 
\frac{1}{[K:\IQ]}\sum_{|\sigma(\psi(P))| >  c_5^{-1}} \log |\sigma(\psi(P))|;
\end{alignat*}
as usual, the sums run over  field embeddings
$\sigma:K\rightarrow\IC$. 
Say $J=\phi(P) \in K$, then $R(J,\psi(P)^{-1})=0$.
We apply Lemma \ref{lem:metric} to $u=J$ and  $v=\psi(P)$ 
to obtain
\begin{equation*}
  \height{\lambda\psi(P)} \le c_7\left(1+\frac{1}{[K:\IQ]}\sum_{|\sigma(J-J_0)| <  1} -\log |\sigma(J-J_0)|\right).
\end{equation*}

We already saw that $R$ is not divisible by a linear polynomial in the
variable $V$.   So
 Proposition 5 \cite{BMSprindzhuk} and
$R(J,\psi(P)^{-1})=0$ 
allow
us to bound $\height{J}$ from above linearly in terms of
$\height{(\lambda\psi(P))^{-1}}=\height{\lambda\psi(P)}$. More precisely 
\begin{alignat}1
\nonumber
  \height{J} &\le c_8\left(1+\frac{1}{[K:\IQ]}
\sum_{|\sigma(J-J_0)|<1} -\log |\sigma(J-J_0)|\right)
\end{alignat}
and so
\begin{equation}
\label{eq:hJbound1}
\height{J}\le c_8\left(|\log \epsilon|+\frac{1}{[K:\IQ]}
\sum_{|\sigma(J-J_0)|<\epsilon} -\log |\sigma(J-J_0)|\right)  
\end{equation}
for any $\epsilon \in (0,1/2]$.

The points in $M$ are special, so $J$ is a singular
modulus. 
An elliptic curve attached to $J$ has complex multiplication by an
order with discriminant $\Delta < 0$. 
As in the previous section,  we  will
find an upper bound for $|\Delta|$.

For any embedding $\sigma:K\rightarrow\IC$ we fix $\tau_\sigma\in\cF$ with 
$j(\tau_\sigma)=\sigma(J)$. 
We now proceed as near (\ref{eq:jzlb}) and apply Theorem 2, Chapter 3 \cite{Lang:elliptic}.
If $\epsilon$ is sufficiently small and if $|\sigma(J-J_0)|<\epsilon$, then 
\begin{equation}
\label{eq:ellipticpointcase}
  |\sigma(J-J_0)|\ge  \left\{
  \begin{array}{ll}
   c_9 |\tau_\sigma-\eta_\sigma|^3 &:\text{ if $J_0=0$,}\\
   c_9 |\tau_\sigma-\eta_\sigma|^2 &:\text{ if $J_0=1728$,}\\
   c_9 |\tau_\sigma-\eta_\sigma| &:\text{ else wise.}\\
  \end{array}\right.
\end{equation}
 for some $\eta_\sigma\in\overline\cF$ with 
 $j(\eta_\sigma) = \sigma(J_0)$. 
It is harmless that there are  $2$ choices for $\eta_\sigma$ on the
boundary of $\overline\cF$. 
We note that $\eta_\sigma$ depends only on the base point $J_0$ and
that $\tau_\sigma$ is imaginary quadratic. Thus Lemma
\ref{lem:dioapprox} and the height bound for $\tau_\sigma$
in Lemma \ref{lem:htaubound} yield
$\log |\sigma(J-J_0)| \ge -c_{10}\log|\Delta|$. 
We use this inequality and (\ref{eq:hJbound1}) to bound
\begin{equation*}
  \height{J}\le c_{11}\left(\log|\epsilon| + {\log|\Delta|}
  \frac{\#\{\sigma:K\rightarrow\IC;\,\,
|\sigma(J-J_0)|< \epsilon \}}{[K:\IQ]}\right)
\end{equation*}
for all $\epsilon\in (0,1/2]$. 

The rest of the proof resembles the proof of Theorem \ref{thm:junit}.
Indeed, we may assume that $[\IQ(J):\IQ]$ is sufficiently large and 
as in Lemma \ref{lem:equi} we use equidistribution to prove
that $[K:\IQ]^{-1}\#\{\sigma:K\rightarrow\IC;\,\,
  |\sigma(J-J_0)|<\epsilon \}$ 
is bounded from above linearly by a fixed power, derived from
(\ref{eq:ellipticpointcase}),  of $\epsilon$.
Finally,  we again use the height lower in Lemma \ref{lem:colmez}
to fix an appropriate $\epsilon$ 
which leads to a bound on $|\Delta|$.
As before, this leaves us with only finitely many possibilities for $J=\phi(P)$. 
\qed

\bibliographystyle{amsplain}
\def\cprime{$'$}
\providecommand{\bysame}{\leavevmode\hbox to3em{\hrulefill}\thinspace}
\providecommand{\MR}{\relax\ifhmode\unskip\space\fi MR }
\providecommand{\MRhref}[2]{%
  \href{http://www.ams.org/mathscinet-getitem?mr=#1}{#2}
}
\providecommand{\href}[2]{#2}


\begin{thebibliography}{10}

\bibitem{BaIhRu}
M.~Baker, S.~Ih, and R.~Rumely, \emph{A finiteness property of torsion points},
  Algebra Number Theory \textbf{2} (2008), no.~2, 217--248.

\bibitem{BiMaZa13}
Y.~Bilu, D.~Masser, and U.~Zannier, \emph{An effective ``theorem of {A}ndr\'e''
  for {$CM$}-points on a plane curve}, Math. Proc. Cambridge Philos. Soc.
  \textbf{154} (2013), no.~1, 145--152.

\bibitem{BMSprindzhuk}
Y.F. Bilu and D.W. Masser, \emph{{A} quick proof of {S}prindzhuk's
  decomposition theorem}, Bolyai Soc. Math. Stud. \textbf{15} (2006).

\bibitem{BG}
E.~Bombieri and W.~Gubler, \emph{{H}eights in {D}iophantine {G}eometry},
  Cambridge University Press, 2006.

\bibitem{CU:Equidistribution}
L.~Clozel and E.~Ullmo, \emph{\'{E}quidistribution des points de {H}ecke},
  Contributions to automorphic forms, geometry, and number theory, Johns
  Hopkins Univ. Press, Baltimore, MD, 2004, pp.~193--254.

\bibitem{Colmez}
P.~Colmez, \emph{Sur la hauteur de {F}altings des vari{\'e}t{\'e}s
  ab{\'e}liennes {\`a} multiplication complexe}, Compositio Math. \textbf{111}
  (1998), no.~3, 359--368.

\bibitem{DHK}
S.~David and N.~Hirata-Kohno, \emph{Linear forms in elliptic logarithms}, J.
  Reine Angew. Math. \textbf{628} (2009), 37--89.

\bibitem{Duke:hyperbolic}
W.~Duke, \emph{Hyperbolic distribution problems and half-integral weight
  {M}aass forms}, Invent. Math. \textbf{92} (1988), no.~1, 73--90.

\bibitem{habegger:wbhc}
P.~Habegger, \emph{Weakly bounded height on modular curves}, Acta Math.
  Vietnam. \textbf{35} (2010), no.~1, 43--69.

\bibitem{Lang:elliptic}
S.~Lang, \emph{Elliptic {F}unctions}, Springer, 1987.

\bibitem{NT}
Y.~Nakkajima and Y.~Taguchi, \emph{A generalization of the {C}howla-{S}elberg
  formula}, J. Reine Angew. Math. \textbf{419} (1991), 119--124.

\bibitem{Serre:MordellWeil}
J.-P. Serre, \emph{Lectures on the {M}ordell-{W}eil theorem}, third ed.,
  Aspects of Mathematics, Friedr. Vieweg \& Sohn, Braunschweig, 1997.

\bibitem{Silverman:HEC}
J.H. Silverman, \emph{Heights and {E}lliptic {C}urves}, Arithmetic Geometry
  (Gary Cornell and Joseph~H. Silverman, eds.), Springer, 1986, pp.~253--265.

\end{thebibliography}


\end{document}